\documentclass[12pt,english,a4paper]{smfart}

\usepackage[T1]{fontenc}
\usepackage{lmodern}
\usepackage{smfthm}
\usepackage[headings]{fullpage}
\usepackage{amssymb,amscd}

\newcommand{\Qp}{\mathbf{Q}_p}
\newcommand{\Cp}{\mathbf{C}_p}
\newcommand{\ZZ}{\mathbf{Z}}
\newcommand{\OO}{\mathcal{O}}
\newcommand{\MM}{\mathfrak{m}}

\newcommand{\Fpbar}{\overline{\mathbf{F}}_p}
\renewcommand{\phi}{\varphi}
\newcommand{\phiq}{\varphi_q}
\renewcommand{\geq}{\geqslant}
\renewcommand{\leq}{\leqslant} 
\newcommand{\bigO}{\mathrm{O}}
\newcommand{\calF}{\mathcal{F}}
\newcommand{\Gal}{\mathrm{Gal}}
\newcommand{\End}{\mathrm{End}}

\newcommand{\LT}{\mathrm{LT}}
\newcommand{\Emb}{\Sigma}
\newcommand{\Fil}{\mathrm{Fil}}
\newcommand{\val}{\mathrm{val}}
\newcommand{\unr}{\mathrm{unr}}
\newcommand{\Log}{\mathrm{L}_{\mathcal{F}}}
\newcommand{\atplus}{\tilde{\mathbf{A}}^+}
\newcommand{\etplus}{\widetilde{\mathbf{E}}^+}
\newcommand{\bdr}{\mathbf{B}_{\mathrm{dR}}}  
\newcommand{\bcris}{\mathbf{B}_{\mathrm{cris}}}  
\newcommand{\dcroc}[1]{[\![ #1 ]\!]}
\newcommand{\ubar}{\overline{u}}
\newcommand{\wideg}{\operatorname{wideg}}

\author{Laurent Berger}
\address{UMPA de l'ENS de Lyon \\
UMR 5669 du CNRS \\ IUF}
\email{laurent.berger@ens-lyon.fr}
\urladdr{perso.ens-lyon.fr/laurent.berger/}

\title{Lubin's conjecture for full $p$-adic dynamical systems}

\alttitle{La conjecture de Lubin pour les syst\`emes dynamiques $p$-adiques pleins}

\date{\today}

\begin{document}

\begin{abstract}
We give a short proof of a conjecture of Lubin concerning certain families of $p$-adic power series that commute under composition. We prove that if the family is \emph{full} (large enough),  there exists a Lubin-Tate formal group such that all the power series in the family are endomorphisms of this group. The proof uses ramification theory and some $p$-adic Hodge theory.
\end{abstract}

\begin{altabstract}
Nous donnons une d\'emonstration courte d'une conjecture de Lubin concernant certaines familles de s\'eries formelles $p$-adiques qui commutent pour la composition. Nous montrons que si la famille est \emph{pleine} (assez grosse), il existe un groupe formel de Lubin-Tate tel que toutes les s\'eries de la famille sont des endomorphismes de ce groupe. La d\'emonstration utilise la th\'eorie de la ramification et un peu de th\'eorie de Hodge $p$-adique.
\end{altabstract}



\maketitle

\setlength{\baselineskip}{18pt}

\section*{Introduction}\label{intro}

Let $K$ be a finite extension of $\Qp$, and let $\OO_K$ be its ring of integers. In \cite{L94}, Lubin studied \emph{$p$-adic dynamical systems}, namely families of elements of $T \cdot \OO_K \dcroc{T}$ that commute under composition, and remarked that ``experimental evidence seems to suggest that for an invertible series to commute with a noninvertible series, there must be a formal group somehow in the background''. This observation has motivated the work of a number of people (Hsia, Laubie, Li, Movaheddi, Salinier, Sarkis, Specter, ...) who proved various results in that direction. The purpose of this note is to give a proof of a special case of the above observation, which is referred to as ``Lubin's conjecture'' in \S 3.1 of \cite{GS10}. Let us consider a family $\calF$ of commuting power series $F(T) \in T \cdot \OO_K \dcroc{T}$. We say that such a family is \emph{full} if for all $\alpha \in \OO_K$ there exists $F_\alpha(T) \in \calF$ such that $F_\alpha'(0) = \alpha$ and if $\wideg(F_\pi(T)) = q$, where $\wideg(F(T))$ denotes the Weierstrass degree of $F(T)$, $\pi$ is any uniformizer of $\OO_K$ and $q$ is the cardinality of the residue field of $\OO_K$.

\begin{enonce*}{Theorem}
If $\calF$ is a full family of commuting power series, there exists a Lubin-Tate formal group $G$ such that $F_\alpha(T) \in \End(G)$ for all $\alpha \in \OO_K$.
\end{enonce*}

This result already appears as theorem 2 of \cite{HL15}. Our proof is however considerably shorter than that of ibid., and does not use the theory of the field of norms. It is very similar to that of the main result of \cite{JS15}, which treats the case $K=\Qp$. The main ingredients are  ramification theory and some $p$-adic Hodge theory. In order to simplify the use of $p$-adic Hodge theory, we assume that $K$ is a Galois extension of $\Qp$. 

\section{$p$-adic dynamical systems}
\label{serlub}

In this note, we consider a set $\calF = \{ F_\alpha(T) \}_{\alpha \in \OO_K}$ of power series $F_\alpha(T) \in T \cdot \OO_K \dcroc{T}$ such that $F_\alpha'(0) = \alpha$ and $F_\alpha \circ F_\beta (T) = F_\beta \circ F_\alpha (T)$ whenever $\alpha$, $\beta \in \OO_K$. Recall that $\pi$ is a uniformizer of $\OO_K$, and that $q$ is the cardinality of the residue field $k$ of $\OO_K$. If $F(T)$ is a power series and $n \geq 0$, we denote by $F^{\circ n}(T)$ the $n$-th fold iteration $F \circ \cdots \circ F(T)$. If $F(T)$ has an inverse for the composition, this definition extends to $n \in \ZZ$. Recall that the \emph{Weierstrass degree} $\wideg(F(T))$ of $F(T) = \sum_{i=1}^{+\infty} f_i T^i \in T \cdot \OO_K \dcroc{T}$ is the smallest integer $i$ such that $f_i \in \OO_K^\times$. By the Weierstrass preparation theorem, if $\wideg(F) \neq +\infty$, then $F$ has $\wideg(F)$ zeroes in $\MM_{\Cp}$.

\begin{prop}
\label{serg}
There exists a power series $G(T) \in T \cdot k \dcroc{T}$ and an integer $d \geq 1$ such that $G'(0) \in k^\times$ and $F_\pi(T) \equiv G(T^{p^d})$.
\end{prop}

\begin{proof}
See (the proof of) theorem 6.3 and corollary 6.2.1 of \cite{L94}.
\end{proof}

\begin{prop}
\label{lublog}
There exists a power series $\Log(T) \in K \dcroc{T}$ such that
\begin{enumerate}
\item $\Log(T) = T + \bigO(T^2)$;
\item $\Log(T)$ converges on the open unit disk;
\item $\Log \circ F_\alpha(T) = \alpha \cdot \Log(T)$ for all $\alpha \in \OO_K$.
\end{enumerate}
\end{prop}

\begin{proof}
See propositions 1.2 and 2.2 of \cite{L94} for the construction of a unique power series $\Log(T)$ that satisfies (1), (2) and (3) for $\alpha$ a uniformizer of $\OO_K$. If $\beta \in \OO_K \setminus \{ 0 \}$, then $\beta^{-1} \cdot \Log \circ F_\beta$ also satisfies (1), (2) and (3) for $\alpha$ as above, so that $\Log \circ F_\beta(T) = \beta \cdot \Log(T)$ for all $\beta \in \OO_K$.
\end{proof}

The hypothesis that $\calF$ is full implies that $p^d=q$, so that $\wideg(F_\pi(T)) = q$. For $n \geq 1$, let $\Lambda_n$ denote the set of $u \in \MM_{\Cp}$ such that $F_\pi^{\circ n}(u) =0$ and $F_\pi^{\circ n-1}(u) \neq 0$ and let $\Lambda_\infty = \cup_{n \geq 1} \Lambda_n$. Proposition \ref{serg} implies that $F_\pi'(T)/  \pi$ is a unit of $\OO_K\dcroc{T}$, so that the roots of $F_\pi^{\circ n}(T)$ are simple for all $n \geq 1$. The set $\Lambda_n$ therefore has $q^{n-1}(q-1)$ elements.

The series $F_\alpha(T)$ is invertible if $\alpha \in \OO_K^\times$ so that in this case, $F_\alpha(z) = 0$ if and only if $z=0$. If $u \in \Lambda_n$ and $\alpha \in \OO_K^\times$, then $F_\pi^{\circ n} \circ F_\alpha(u) = F_\alpha \circ F_\pi^{\circ n}(u) = 0$ and $F_\pi^{\circ n-1} \circ F_\alpha(u) = F_\alpha \circ F_\pi^{\circ n-1}(u) \neq 0$  so that the action of $F_\alpha(T)$ permutes the elements of $\Lambda_n$.

Let $K_n=K(\Lambda_n)$, so that $\Lambda_i \subset K_n$ if $i \leq n$, and let $K_\infty = \cup_{n \geq 1} K_n$. If $\alpha \in \OO_K^\times$, let $n(\alpha)$ be the largest integer $n \geq 0$ such that $\alpha \in 1+\pi^n \OO_K$.

\begin{prop}
\label{sizelam}
If $n \geq 1$ and $u \in \Lambda_n$, then
\begin{enumerate}
\item $F_\alpha(u) = u$ if and only if $n(\alpha) \geq n$;
\item If $n(\alpha) =n$, then $\wideg(F_\alpha(T)-T) = q^n$;
\item $\Lambda_n = \{ F_\alpha(u) \}_{\alpha \in \OO_K^\times}$.
\end{enumerate}
\end{prop}

\begin{proof}
If $n=1$ and $F_\alpha(u) = u$, then $u$ is a root of $F_\alpha(T)-T = (\alpha-1)T + \bigO(T^2)$, so that $\alpha - 1 \in \pi \OO_K$. This implies that $\{ F_\alpha(u) \}_{\alpha \in \OO_K^\times}$ has at least $q-1$ distinct elements. These elements are all roots of $F_\pi(T)/T$, whose $\wideg$ is $q-1$, so $\{ F_\alpha(u) \}_{\alpha \in \OO_K^\times}$ has precisely $q-1$ elements. These elements all have valuation $1/(q-1)$, and if $n(\alpha)=1$, the Newton polygon of $F_\alpha(T)-T$ starts at the point $(1,1)$, so that it can have only one segment, and $\wideg(F_\alpha(T)-T) = q$. This implies the proposition for $n=1$.

Assume now that the proposition holds up to some $n \geq 1$ and take $u \in \Lambda_{n+1}$. If $n(\alpha) \leq n$, then $F_\alpha(T)-T$ has at most $q^n$ roots by (2), contained in $\Lambda_0 \cup \hdots \cup \Lambda_n$ by (1). Therefore $F_\alpha(u) = u$ implies $n(\alpha) \geq n+1$. The set $\{ F_\alpha(u) \}_{\alpha \in \OO_K^\times}$ therefore has at least $q^n(q-1)$ distinct elements, all of them roots of $F_\pi^{\circ n+1}(T)/F_\pi^{\circ n}(T)$. 

This implies that $\{ F_\alpha(u) \}_{\alpha \in \OO_K^\times}$ has exactly $q^n(q-1)$ elements. If $n(\alpha)=n+1$, the Newton polygon of $F_\alpha(T)-T$ starts at the point $(1,n+1)$, with $n+1$ segments of height one and slopes $-1/q^k(q-1)$ with $0 \leq k \leq n$, so that it reaches the point $(q^{n+1},0)$ and hence $\wideg(F_\alpha(T)-T) = q^{n+1}$. This implies the proposition for $n+1$.
\end{proof}

\begin{coro}
\label{extab}
The field $K_\infty$ is an abelian totally ramified extension of $K$, and if $g \in \Gal(K_\infty/K)$, there is a unique $\eta(g) \in \OO_K^\times$ such that $g(u) = F_{\eta(g)}(u)$ for all $u \in \Lambda_\infty$. 

The map $\eta : \Gal(K_\infty/K) \to \OO_K^\times$ is an isomorphism.
\end{coro}

\begin{proof}
Take $u \in \Lambda_n$ and $\alpha \in \OO_K^\times$. As we have seen above, $F_\alpha(u) \in \Lambda_n$, so that the map $u \mapsto F_\alpha(u)$ induces a field automorphism of $K(u)$. By (3) of proposition \ref{sizelam}, this implies that $K_n = K(u)$ and that every element of $\Gal(K_n/K)$ comes from $u \mapsto F_\alpha(u)$ for some $\alpha \in \OO_K^\times$. The extension $K_n/K$ is therefore abelian, and so is $K_\infty/K$. Since $K_n=K(u)$, the extension $K_n/K$ is totally ramified, and so is $K_\infty/K$.

The map $\eta$ is surjective because every $F_\alpha(T)$ gives rise to an automorphism of $K_\infty$, and it is injective because if $\eta(g)=1$, then $g(u)=u$ for all $u \in \Lambda_\infty$ so that $g=1$.
\end{proof}

In order to prove our main theorem, we study the $p$-adic periods of $\eta$. Corollary \ref{extab} and local class field theory imply that the extension $K_\infty/K$ is attached to a uniformizer $\varpi$ of $\OO_K$. Let $\chi_\varpi : G_K \to \OO_K^\times$ denote the corresponding Lubin-Tate character.

\section{$p$-adic Hodge theory}
\label{charid}

Let $R$ be the $p$-adic completion of $\varinjlim_{n  \geq 1} \OO_K \dcroc{X_n}$ where $\OO_K \dcroc{X_n}$ is seen as a subring of $\OO_K \dcroc{X_{n+1}}$ via the identification $X_n=F_\pi(X_{n+1})$ (this ring is defined in \cite{JS15}, where it is denoted by $A_\infty$). We define an action of $G_K$ on $R$ by $g(H(X_n)) = H ( F_{\eta(g)}(X_n))$. This is well-defined since $F_\pi \circ F_{\eta(g)} (T) = F_{\eta(g)} \circ F_\pi (T)$. We have $R/\pi R = \varinjlim_{n  \geq 1} k \dcroc{X_n}$. 

\begin{lemm}
\label{rperf}
The ring $R/\pi R$ is perfect.
\end{lemm}

\begin{proof}
Let $G(T)$ be as in lemma \ref{serg}. The fact that $X_n=F_\pi(X_{n+1})$ implies that $G^{\circ n}(X_n) = G^{\circ n+1}(X_{n+1})^q$. Since $G'(0) \in k^\times$, we have $k \dcroc{T} = k \dcroc{G(T)}$ and therefore 
\[ R/\pi R = \varinjlim_{G^{\circ n}(X_n) = G^{\circ n+1}(X_{n+1})^q} k \dcroc{G^{\circ n}(X_n)} \]
is perfect.
\end{proof}

Let $\etplus = \varprojlim_{x \mapsto x^q} \OO_{\Cp} / \pi$. Choose a sequence $\{u_n\}_{n \geq 1}$ with $u_n \in \Lambda_n$ and $F_\pi(u_{n+1}) = u_n$. This sequence gives rise to a map $i : R / \pi R \to \etplus$, determined by the requirement that $i(X_n) = (G^{\circ 1}(\ubar_n),G^{\circ 2}(\ubar_{n+1}),\hdots)$. The definition of the action of $G_K$ on $R$ and corollary \ref{extab} imply that $i$ is $G_K$-equivariant.

\begin{lemm}
\label{inji}
The map $i : R / \pi R \to \etplus$ is injective.
\end{lemm}

\begin{proof}
It is enough to show that $i : k \dcroc{X_n} \to \etplus$ is injective. If it was not, there would be a nonzero polynomial $P(T) \in k[T]$ such that $P(i(X_n)) = 0$, and $i(X_n) = (G^{\circ 1}(\ubar_n),G^{\circ 2}(\ubar_{n+1}),\hdots)$ 
would then belong to $\Fpbar$, which is clearly not the case.
\end{proof}

Let $K_0 = \Qp^{\unr} \cap K$ and let $\atplus = \OO_K \otimes_{\OO_{K_0}} W(\etplus)$ (see \cite{FPP}; note that $\atplus$ usually denotes $W(\etplus)$, and is denoted by $A_{\mathrm{inf}}$ in ibid.). We have $R = \OO_K \otimes_{\OO_{K_0}} W(R/\pi R)$ since $R$ is a strict $\pi$-ring, and by the functoriality of Witt vectors, the map $i$ extends to an injective and $G_K$-equivariant map $i : R \to \atplus$. We write $x$ instead of $i(X_1) \in \atplus$. The $G_K$-equivariance of $i$ implies that $g(x) = F_{\eta(g)}(x)$.

Let $\bcris^+$ and $\bdr$ be some of Fontaine's rings of periods. Recall that $\bdr$ is a field, that there is a Frobenius map $\phi$ on $\bcris^+$, a filtration $\{\Fil^i \bdr\}_{i \in \ZZ}$ on $\bdr$, and an injective map $K \otimes_{K_0} \bcris^+ \to \bdr$. There is also an action of $G_K$ on $\bcris^+$ and $\bdr$ compatible with the above structure, and $\bdr^{G_K} = K$. Let $\phiq=\phi^f$ on $\bcris^+$, where $q=p^f$, extended by $K$-linearity to $K \otimes_{K_0} \bcris^+$. We refer to \cite{FPP} and \cite{FST} for the properties of these objects. Let $\Emb = \Gal(K/\Qp)$. If $\tau \in \Emb$, choose a $n(\tau) \in \ZZ_{\geq 0}$ such that $\tau {\mid_{K_0}} = \phi^{n(\tau)}$. The map $\tau \otimes \phi^{n(\tau)} : K \otimes_{K_0} \bcris^+ \to K \otimes_{K_0} \bcris^+$ is then well-defined and commutes with $\phiq$ and the action of $G_K$. 

We say that a character $\lambda : G_K \to \OO_K^\times$ is \emph{crystalline positive} if there exists a nonzero $z \in K \otimes_{K_0} \bcris^+$ such that $g(z) = \lambda(g) \cdot z$ for all $g \in G_K$. The following proposition summarizes the input that we need from the $p$-adic Hodge theory of characters.

\begin{prop}
\label{cryschar}
A character $\lambda : G_K \to \OO_K^\times$ that factors through $\Gal(K_\infty/K)$ is crystalline positive if and only if $\lambda= \prod_{\tau \in \Emb} \tau \circ \chi_\varpi^{h_\tau}$ with $h_\tau \in \ZZ_{\geq 0}$.

If $t_\varpi \in K \otimes_{K_0} \bcris^+$ is such that $g(t_\varpi) = \chi_\varpi(g) \cdot t_\varpi$ for all $g \in G_K$, then $t_\varpi \in \Fil^1 \bdr$ and $\phiq(t_\varpi) = \varpi \cdot t_\varpi$.
\end{prop}

\begin{proof}[Sketch of proof]
If $\lambda : G_K \to \OO_K^\times$ is a crystalline positive character and $h_\tau \in \ZZ_{\geq 0}$ denotes the \emph{Hodge-Tate weight} of $\lambda$ at $\tau \in \Emb$, then $\lambda \cdot \prod_{\tau \in \Emb} \tau \circ \chi_\varpi^{-h_\tau}$ is crystalline and has Hodge-Tate weight zero at all $\tau \in \Emb$ so that it is unramified, and therefore trivial if $\lambda$ factors through $\Gal(K_\infty/K)$, since $K_\infty/K$ is totally ramified.

Let $\omega_E$ and $t_E$ be the elements constructed in \S 9.2 and \S 9.3 of \cite{CEB} (with $E=K$ and $\pi_E=\varpi$). We have $t_E \in K \otimes_{K_0} \bcris^+$ and $\phiq(t_E) = \varpi \cdot t_E$ and $t_E \in \Fil^1 \bdr$ (proposition 9.10 of ibid). If $g \in G_K$, then (in the notation of ibid and where $[\,\cdot\,]_{\LT}$ denotes the endomorphisms of the Lubin-Tate group attached to $\varpi$) we have $g(\omega_E) = [\chi_\varpi(g)]_{\LT} (\omega_E)$ and therefore $g(t_E) = g (F_E(\omega_E)) = F_E (g(\omega_E)) = F_E \circ [\chi_\varpi(g)]_{\LT} (\omega_E) = \chi_\varpi(g) \cdot F_E(\omega_E) = \chi_\varpi(g) \cdot t_E$ since $F_E$ is the logarithm of the Lubin-Tate group attached to $\varpi$. If $t_\varpi \in K \otimes_{K_0} \bcris^+$ is such that $g(t_\varpi) = \chi_\varpi(g) \cdot t_\varpi$ for all $g \in G_K$, then $t_\varpi / t_E \in \bdr^{G_K} = K$, and this implies the rest of the proposition.
\end{proof}

Recall that $\Log(T) \in K \dcroc{T}$ is the logarithm attached to $\calF$. Since $\Log(T)$ converges on the open unit disk, we can view $\Log(x)$ as an element of $K \otimes_{K_0} \bcris^+$. 

\begin{prop}
\label{etacris}
The character $\eta : G_K \to \OO_K^\times$ is crystalline positive.
\end{prop}

\begin{proof}
If $g \in G_K$, then $g( \Log(x) ) = \Log (g(x)) = \Log(F_{\eta(g)}(x))  = \eta(g) \cdot \Log(x)$.
\end{proof}

\begin{coro}
\label{logex}
We have $\Log(x) =  \beta \cdot \prod_{\tau \in \Emb} (\tau \otimes \phi^{n(\tau)} )(t_\varpi)^{h_\tau}$ where $h_\tau \in \ZZ_{\geq 0}$ and $\beta \in K^\times$.
\end{coro}

\begin{proof}
This follows from the facts that $\eta= \prod_{\tau \in \Emb} \tau \circ \chi_\varpi^{h_\tau}$, that $\chi_\varpi(g) = g(t_\varpi)/t_\varpi$ and that $\bdr^{G_K} = K$.
\end{proof}

\begin{prop}
\label{philog}
We have $\phiq(\Log(x)) = \mu \cdot \Log(x)$ for some $\mu \in \pi \OO_K$.
\end{prop}

\begin{proof}
Corollary \ref{logex} and proposition \ref{cryschar} imply the proposition with $\mu = \prod_\tau \tau(\varpi)^{h_\tau}$, and not all $h_\tau$ can be equal to $0$ since $\eta \neq 1$.
\end{proof}

\begin{coro}
\label{phiexp}
We have $\phiq(x) = F_\mu(x)$.
\end{coro}

\begin{proof}
Proposition \ref{philog} implies that $\Log(\phiq(x)) = \Log(F_\mu(x))$. We would like to apply $\Log^{\circ -1}(T)$ but we have to mind the convergence and need to proceed as follows. Since $\eta$ is nontrivial, there is a $\tau \in \Emb$ such that $h_{\tau^{-1}} \geq 1$. We have 
\[ (\tau \otimes \phi^{n(\tau)}) (\Log(\phiq(x))) = (\tau \otimes \phi^{n(\tau)}) (\Log(F_\mu(x))) \] 
in $K \otimes_{K_0} \bcris^+$ and $h_{\tau^{-1}} \geq 1$ now implies that $(\tau \otimes \phi^{n(\tau)}) (\Log(\phiq(x)))$ is divisible by $t_\varpi$ so that by proposition \ref{cryschar}, it belongs to  $\Fil^1 \bdr$. We can then apply $\Log^{\tau \circ -1}(T)$ in $\bdr$ and get that $(\tau \otimes \phi^{n(\tau)}) (\phiq(x)) = (\tau \otimes \phi^{n(\tau)}) (F_\mu(x))$ in $\bdr$. This equality also holds in $\atplus$, so that $\phiq(x) = F_\mu(x)$.
\end{proof}

\begin{theo}
\label{lubconj}
There is a Lubin-Tate formal group $G$ such that $F_\alpha(T) \in \End(G)$ for all $\alpha \in \OO_K$.
\end{theo}

\begin{proof}
By corollary \ref{phiexp}, we have $\phiq(x) = F_\mu(x)$. This implies that $F_\mu(T) \equiv T^q \bmod{\pi \OO_K \dcroc{T}}$. The Weierstrass degree of $F_\mu(T)$ is $q^{\val(\mu)}$ so that $\val(\mu)=1$ and $F_\mu(T)$ is a Lubin-Tate power series. By \cite{LT65}, there is a Lubin-Tate formal group $G$ such that $F_\mu(T) \in \End(G)$. Since $F_\alpha(T)$ commutes with $F_\mu(T)$, we also have $F_\alpha(T) \in \End(G)$ for all $\alpha \in \OO_K$.
\end{proof}

\begin{rema}
\label{wrapup}
We have $\mu=\varpi$ and $\eta=\chi_\varpi$. Indeed, the extension $K_\infty/K$ is generated by the torsion points of $G$, and is therefore attached to $\mu$ by local class field theory, so that $\mu=\varpi$. This in turn implies that $\eta=\chi_\varpi$.
\end{rema}

\providecommand{\bysame}{\leavevmode ---\ }
\providecommand{\og}{``}
\providecommand{\fg}{''}
\providecommand{\smfandname}{\&}
\providecommand{\smfedsname}{\'eds.}
\providecommand{\smfedname}{\'ed.}
\providecommand{\smfmastersthesisname}{M\'emoire}
\providecommand{\smfphdthesisname}{Th\`ese}

\end{document}